\documentclass[a4paper,11pt,leqno]{amsart}

\usepackage{a4wide}
\usepackage{amsmath}
\usepackage[italian, english]{babel}
\usepackage{amsfonts}
\usepackage{amssymb}
\usepackage{dsfont}
\usepackage{mathrsfs}
\usepackage{graphicx}
\usepackage{fancyhdr}
\usepackage{amsthm}
\usepackage{empheq}
\usepackage{cases}
\usepackage[all]{xy}
\usepackage{stmaryrd}
\usepackage[colorlinks, citecolor=blue, linkcolor=red]{hyperref}
\usepackage{mathrsfs}

\usepackage{mathtools}
\def\Yf{\mathcal{Y}_{f}}

\def\SS{{{\mathbb S}}}

\def\RR{{\mathbb R}}
\def\HH{{\mathbb H}}

\def\gt{\widetilde{g}}

\def\gt{\widetilde{g}}

\newcommand{\ricc}{\operatorname{Ric}}

\newcommand{\pa}[1]{{\left(#1\right)}}                  % tra tonde
\newcommand{\abs}[1]{{\left|#1\right|}}                 % valore assoluto
\newcommand{\eps}{\varepsilon}                           % epsilon
\renewcommand{\tilde}[1]{\widetilde{#1}}

\newtheorem{ackn}{Acknowledgments\!}

\newtheorem{theorem}{\textbf{Theorem}}[section]
\newtheorem{lemma}[theorem]{\textbf{Lemma}}
\newtheorem{proposition}[theorem]{\textbf{Proposition}}

\newtheorem{definition}[theorem]{\textbf{Definition}}

\theoremstyle{remark}

\numberwithin{equation}{section}

\def\neweq#1{\begin{equation}\label{#1}}
\def\endeq{\end{equation}}
\def\eq#1{(\ref{#1})}
\def\R{\mathbb{R}}
\def\eps{\varepsilon}

\title[A conformal Yamabe problem with potential on the euclidean space]
{A conformal Yamabe problem with potential\\ on the euclidean space}

\author[G. Catino]{Giovanni Catino}
\address[Giovanni Catino]{Dipartimento di Matematica, Politecnico di Milano, Piazza Leonardo da Vinci 32, 20133 Milano, Italy}
\email[]{giovanni.catino@polimi.it}

\author[F. Gazzola]{Filippo Gazzola}
\address[Filippo Gazzola]{Dipartimento di Matematica, Politecnico di Milano, Piazza Leonardo da Vinci 32, 20133 Milano, Italy}
\email[]{filippo.gazzola@polimi.it}

\author[P. Mastrolia]{Paolo Mastrolia}
\address[Paolo Mastrolia]{Dipartimento di Matematica, Universit\`{a} degli Studi di Milano, Via Saldini 50, 20133 Italy.}
\email[]{paolo.mastrolia@unimi.it}

\begin{document}

\begin{abstract}
We consider, in the Euclidean setting, a conformal Yamabe-type equation related to a potential generalization of the classical constant scalar curvature problem and which naturally arises in the study of Ricci solitons structures. We prove existence and nonexistence results, focusing on the radial case, under some general hypothesis on the potential.
\end{abstract}

\maketitle

\begin{center}

\noindent{\it Key Words: Yamabe problem, conformal problems, constant scalar curvature, ordinary differential equations}

\medskip

\centerline{\bf AMS subject classification:  53C20, 53C25, 34A34}

\end{center}

\

\

\section{introduction}

In \cite{cm} the first and the third author considered "potential" generalizations of some canonical metrics on smooth complete Riemannian manifolds. In this paper we focus our attention on one of those classes, namely {\em $f$-Yamabe metrics}. We recall that, given a $n$-dimensional Riemannian manifold $(M,g)$, where $g$ is the metric, and a smooth function $f\in C^\infty(M)$, we say that the triple $(M,g,f) \in \Yf$ if and only if it satisfies the condition
\begin{equation}\label{yf}
\nabla R = 2\ricc(\nabla f, \cdot)\,,
\end{equation}
where $\ricc$ and $R$ are, respectively, the Ricci and the scalar curvature of $g$ and $\nabla$ denotes the Levi-Civita connection associated to $g$. In a local orthonormal frame $\{e_i\}$, $i=1,\ldots,n$, \eqref{yf} becomes
$$
\nabla_{e_i} R = 2 R_{ij} \nabla_{e_j} f \,,
$$
where $R_{ij}=\ricc(e_i,e_j)$. Note that we are using the Einstein summation convention over repeated indices. This equation is a meaningful generalization of the one for {\em constant scalar curvature metrics} and naturally arises in the study of Ricci solitons structures (for a general overview see \cite{cao}). Moreover, it is clear that any Ricci flat metric satisfy  \eqref{yf}, for any function $f$ and, more in general, any product of a Ricci flat metric with a metric with constant scalar curvature solves \eqref{yf}, for any function $f$ which depends only on the first factor.

In the same spirit of the classical {\em Yamabe problem} it is natural to address the following questions:

\begin{itemize}

\item[(A)] having fixed $f\in C^{\infty}(M)$, does there exist a metric $g$ such that $(M,g,f)\in \Yf$?

\item[(B)] having fixed $f\in C^{\infty}(M)$ and a metric $g$, does there exist a conformal metric $\widetilde{g}$ in the conformal class $[g]$ such that $(M,\widetilde{g},f)\in \Yf$?

\end{itemize}
More generally, one could ask the question

\begin{itemize}

\item[(C)] does there exist a metric $g$ and a smooth function $f\in C^{\infty}(M)$ such that $(M,g,f)\in \Yf$?

\end{itemize}
Clearly the answer to (C) is positive, since it is always possible to construct a (complete) metric with constant (negative) scalar curvature (\cite{aubin} and \cite{blakal}). Furthermore, when $f$ is constant, (B) boils down to the well known Yamabe problem, which is completely solved when $M$ is compact (see e.g.  \cite{leepar}). We will refer to (B) as the {\em conformal $f$-Yamabe problem}. In this paper we consider problem (B) (when $f$ is not constant) on the Euclidean space $\RR^n$ endowed with the standard flat metric $g_{\RR^n}$. In particular, in dimension four, we prove the following:

\begin{theorem} Let $f\in C^{\infty}(\RR^4)$ be a, not constant, radial function satisfying $f'(r)\leq 0$ for all $r>0$. 
Then there exists a conformal metric $\widetilde{g}\in[g_{\RR^n}]$ such that $(\RR^n,\widetilde{g},f)\in \Yf$.
\end{theorem}

\

\section{ODE formulation of the conformal $f$-Yamabe problem}

Let $(M,g)$ be a smooth $n$-dimensional Riemannian manifold, $n\geq 2$, and let $f\in C^\infty(M)$. It is well known (see for instance \cite{catmasmonrig}) that, if $\gt=e^{2w}g \in [g]$ for some $w\in C^\infty(M)$, then the following formulas hold:
\begin{align*}
\widetilde{\ricc} &= \ricc -\pa{n-2}\nabla^2 w+\pa{n-2}dw\otimes dw-(\Delta w)\, g-\pa{n-2}\abs{\nabla w}^2 g \,,\\
\tilde{R} &= e^{-2w}\left(R -2\pa{n-1}\Delta w-\pa{n-1}\pa{n-2}\abs{\nabla w}^2\right) \,
\end{align*}
where $\nabla^2$ is the Hessian and $\Delta=g^{ij}\nabla_{ij}^2$ is the Laplace-Beltrami operator of $g$. A computation shows that $(M,\widetilde{g},f)=(M, e^{2w}g,f) \in \Yf$ if and only if the function $w$ solves the system of PDEs
    \begin{align}\label{cfypg}
      \nabla \Delta w &+ (n-2)\nabla^{2}w(\nabla w,\cdot)-\Big(2\Delta w + (n-2)|\nabla w|^{2} - \frac{1}{n-1}R\Big) \nabla w  -\frac{1}{2(n-1)}\nabla R  \\ \nonumber
     &= -\frac{1}{n-1}\ricc(\nabla f, \cdot) +\frac{n-2}{n-1}\nabla^{2}w(\nabla f,\cdot) + \frac{1}{n-1}\Big(\Delta w +(n-2)|\nabla w|^{2}\Big)\nabla f
\\ \nonumber &- \frac{n-2}{n-1}\langle \nabla w,\nabla f\rangle \nabla w. \,
    \end{align}
In particular, since $(\RR^n, g_{\RR^n})$ is Ricci flat, then $(\RR^n,\widetilde{g},f)=(\RR^n, e^{2w}g_{\RR^n},f)\in \Yf$ if and only if
$w$ solves the system of PDEs
    \begin{align}\label{cfyp}
      \nabla \Delta w &+ (n-2)\nabla^{2}w(\nabla w,\cdot)-\Big(2\Delta w + (n-2)|\nabla w|^{2}\Big) \nabla w    \\ \nonumber
     &= \frac{n-2}{n-1}\nabla^{2}w(\nabla f,\cdot) + \frac{1}{n-1}\Big(\Delta w +(n-2)|\nabla w|^{2}\Big)\nabla f
- \frac{n-2}{n-1}\langle \nabla w,\nabla f\rangle \nabla w. \,
    \end{align}
To fully exploit the symmetries of the Euclidean space, it is reasonable to start our analysis by considering radial solutions $w=w(r)$ of \eqref{cfyp} for a given radial function $f=f(r)$, where $r$ denotes the distance function from the origin. In this case, in standard polar coordinates, one has
\begin{align*}
g_{\RR^n} &=dr^2 + r^2 g_{\SS^{n-1}} \,,\\
d w &= w'(r) dr \quad \text{and}\quad d f = f'(r) dr\,,\\
\nabla dw &= \nabla^2 w = w''(r) dr^2 +r w'(r) g_{\SS^{n-1}} \,,\\
\Delta w &= w''(r)+\frac{n-1}{r}w'(r)\,
\end{align*}
and a computation shows that system \eqref{cfyp} boils down to the following second order nonlinear ODE for the function  $u(r):=w'(r)$
\begin{align}\label{equaz}
u''(r)+\frac{n-1}{r}u'(r)-(n-2)u(r)^3-&\frac{2(n-1)}{r}u(r)^2-\frac{n-1}{r^2}u(r)+(n-4)u(r)u'(r)\\\nonumber&=\left[u'(r)+\frac{u(r)}{r}\right]h(r)\,,
\end{align}
where $h(r):=f'(r)$. Note that if $n=2$, then the cubic term disappears in \eq{equaz}.\par
We then impose the initial conditions
\begin{align}\label{ic}
u(0)=0\, ,\qquad u'(0)=\alpha\neq0
\end{align}
which require some explanation since \eq{equaz} is singular at $r=0$. Assume that $u\in C^2[0,\infty)$ satisfies \eq{ic}, then $u(r)=\alpha r+O(r^2)$ as
$r\to0$ and, in turn,
$$
\frac{n-1}{r}u'(r)-\frac{n-1}{r^2}u(r)=O(1)\, ,\quad\frac{u(r)}{r}=O(1)\, ,\quad\frac{u(r)^2}{r}=o(1)\qquad\mbox{as }r\to0
$$
which shows that, by combining suitably the terms in \eq{equaz}, we obtain finite limits as $r\to0$. The existence and uniqueness of a solution of
\eq{equaz}-\eq{ic} can then be proved rigorously by adapting the arguments of Proposition 1 in \cite{niserrin}: one needs to combine the Ascoli-Arzel\`a
Theorem with the Schauder fixed point Theorem in order to obtain existence of a solution. Then the solution is unique as long as it can be
continued \cite[Proposition 4.2]{fls}.\par
Before stating our existence and nonexistence results, let us discuss heuristically the structure of \eq{equaz}.

\

\section{Heuristic preliminaries}

We first notice that, if $n\ge3$, then there exist exactly two singular (negative) solutions of \eq{equaz} of the type $cr^{-1}$ given by
\neweq{singular}
u_1(r)=-\, \frac1r \, ,\quad u_2(r)=-\, \frac2r\, ,
\endeq
regardless of the explicit form of $h$. This fact suggests that the ``interesting dynamics'' for \eq{equaz} occurs when $u(r)<0$ and global solutions of
\eq{equaz} are more likely to be prevalently negative. If $n=2$ then the functions $u(r)=cr^{-1}$ are singular solutions of \eq{equaz} for any $c\neq0$;
in particular, there exist infinitely many positive singular solutions and the dynamics appears much more chaotic.\par
It is quite useful to consider the two functions defined for all $(r,y)\in\R_+\times\R$:
$$
P(r,y)=\left((n-2)y^2+\frac{2(n-1)}{r}y+\frac{n-1}{r^2}+\frac{h(r)}{r}\right)y\, ,\quad Q(r,y)=h(r)-\frac{n-1}{r}+(4-n)y\, .
$$
Then \eq{equaz} may be written in normal form as
\neweq{poly}
u''(r)=Q\big(r,u(r)\big)u'(r)+P\big(r,u(r)\big)\, .
\endeq
Depending on $h\in C^0[0,\infty)$, we define the two regions
$$
I_h:=\big\{r\ge0;\, (n-2)rh(r)<n-1\big\}\, ,\qquad I^h=\big\{r\ge0;\, (n-2)rh(r)>n-1\big\}\, .
$$
Clearly, $I_h$ contains a right neighborhood of $r=0$ and is therefore nonempty for all $h$, while $I^h$ is empty if $rh(r)\le\frac{n-1}{n-2}$ for
all $r$, in particular if $h(r)\le0$. It is also straightforward that:\par
$\bullet$ if $r\in I^h$, then $P(r,y)=0$ if and only if $y=0$, moreover $P$ has the same sign as $y$;\par
$\bullet$ if $r\in I_h$, then we may write
$$
P(r,y)=(n-2)\left[y+\frac{n-1+\sqrt{n-1-(n-2)rh(r)}}{(n-2)r}\right]\left[y+\frac{n-1-\sqrt{n-1-(n-2)rh(r)}}{(n-2)r}\right]y
$$
and hence $P(r,y)$ vanishes if and only if one of the following facts occurs:
$$
y=0\, ,\quad y=\varphi(r):=-\frac{n-1+\sqrt{n-1-(n-2)rh(r)}}{(n-2)r}\, ,
$$
$$
y=\psi(r):=\frac{1-n+\sqrt{n-1-(n-2)rh(r)}}{(n-2)r}\, .
$$
Note that $\psi(r)>\varphi(r)$ for all $r\in I_h$ but, while $\varphi(r)<0$ for all $r\in I_h$, the sign of $\psi(r)$ may vary and it is the opposite
of the sign of $rh(r)+n-1$; in particular, $\psi(r)<0$ in a right neighborhood of $r=0$.\par
One expects a crucial role for the existence results to be played by the signs of $Q$ and $P$. However, the overall picture is not completely clear.
To see this, consider the trivial case $h\equiv0$ for which the function
\neweq{trivial}
\overline{u}_a(r)=-\, \frac{2ar}{1+ar^2}
\endeq
solves \eq{equaz}-\eq{ic} with $\alpha=-2a$. First notice that if $a<0$ (so that $u'_a(0)>0$), then $u_a$ blows up as $r\to1/\sqrt{|a|}$.
Therefore, $u_a$ is a global solution of \eq{equaz} if and only if $a>0$. Simple computations then show that
$$
\mbox{if }n=3,4,5\mbox{ then }\exists\rho>0\, ,\ \varphi(r)<\overline{u}_1(r)<\psi(r)\ \forall r>\rho\, ,
$$
$$
\mbox{if }n\ge6\mbox{ then }\exists\rho>0\, ,\ \overline{u}_1(r)<\varphi(r)\ \forall r>\rho\, .
$$
These facts are illustrated in Figure \ref{shaded}, where the shaded region is
$$\Gamma:=\{(r,y)\in\R_+\times\R_-;\, \varphi(r)<y<\psi(r)\}\, .$$
In the left picture we see that the graph of $\overline{u}_1$ (thick line) eventually lies inside $\Gamma$ while in the right picture it eventually lies
outside. Therefore, the function $P(r,\overline{u}_1(r))$ does not always have the same sign as $r\to\infty$.\par

\

\begin{figure}[ht]
\begin{center}
\includegraphics[height=40mm, width=75mm]{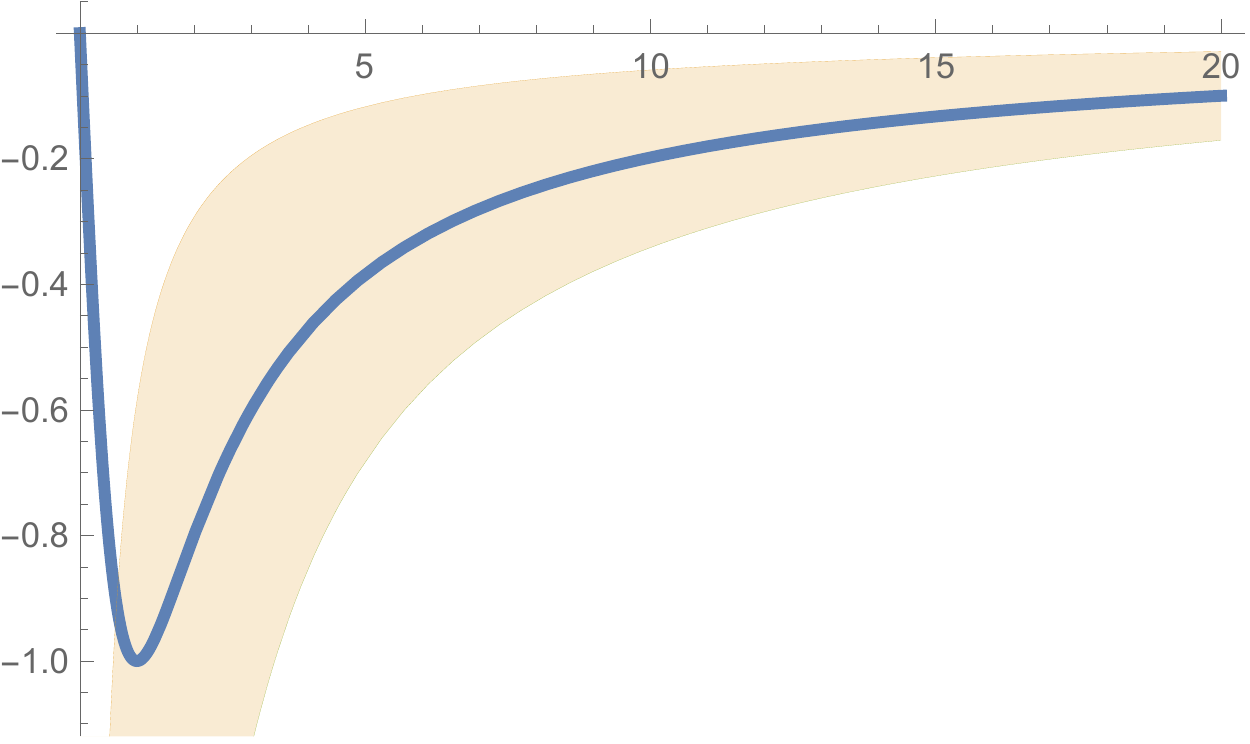}\quad\includegraphics[height=40mm, width=75mm]{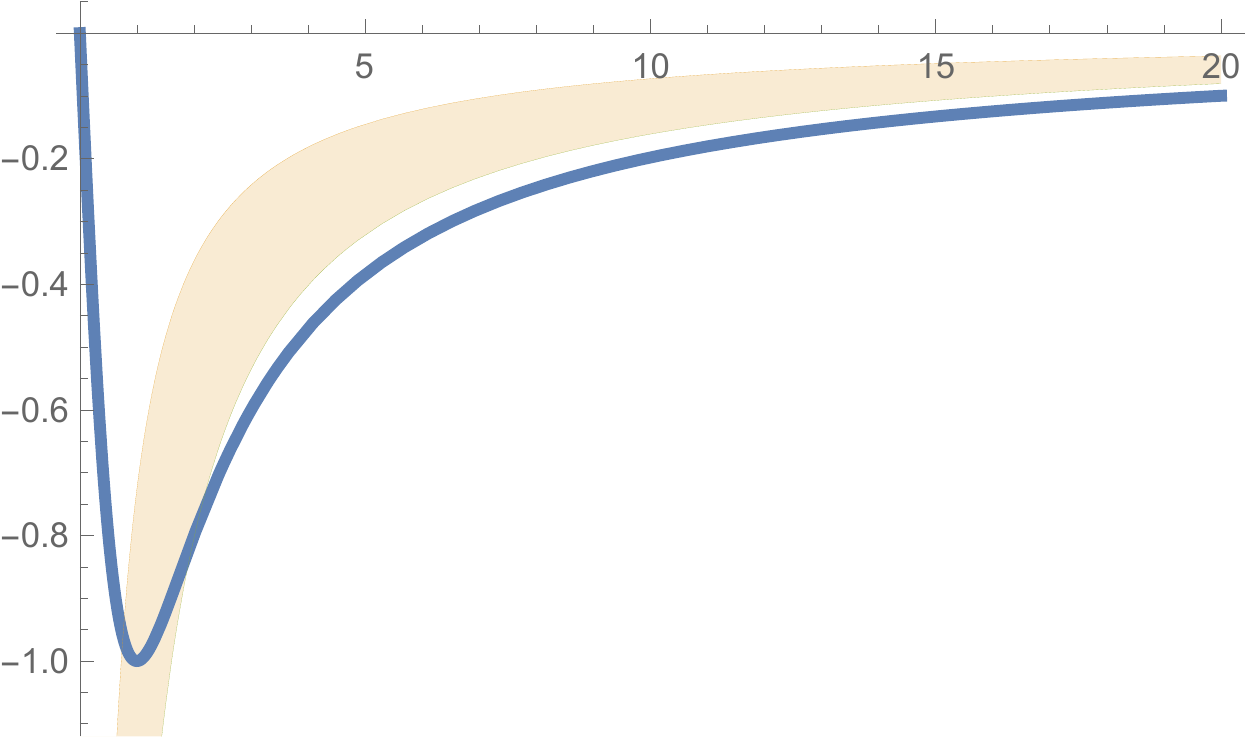}
\caption{Plot of $\overline{u}_1(r)$ in \eq{trivial} (thick line) and of $\Gamma$ (shaded region) when $n=3$ (left) and $n=8$ (right).}\label{shaded}
\end{center}
\end{figure}

\

\section{Nonexistence results}

We can the prove the following (partial) nonexistence results.

\begin{theorem}\label{nonexists}
If $n\ge3$, $h(r)\ge-\frac{n-1}{r}$ for all $r>0$, $\alpha>0$, then the solution of \eqref{equaz}-\eqref{ic} is not global.
\end{theorem}

The proof of Theorem \ref{nonexists} is given in Section \ref{proofnon}. As a by-product, the very same proof enables us to obtain

\begin{theorem}\label{nonexists2}
If $n\ge3$, $h(r)\ge-\frac{n-1}{r}$ for all $r>0$, then any global solution $u$ of \eqref{equaz}-\eqref{ic} is necessarily strictly negative.
\end{theorem}

Indeed, Theorem \ref{nonexists} excludes the existence of positive solutions $u$ satisfying $u'(0)>0$. If $u'(0)<0$ then the solution of \eq{equaz}
is initially negative and, if it becomes positive, one can argue as in Section \ref{proofnon} in order to show finite space blow up.\par
Concerning nonexistence of negative solutions, a weaker result holds. First of all, we put together the three static terms
\begin{align}\label{static}
(n-2)u(r)^3+&\frac{2(n-1)}{r}u(r)^2+\frac{n-1}{r^2}u(r)\\\nonumber&=(n-2)\left(u(r)+\frac{n-1+\sqrt{n-1}}{(n-2)r}\right)
\left(u(r)+\frac{n-1-\sqrt{n-1}}{(n-2)r}\right)u(r)\, .
\end{align}
This shows that the static term changes sign whenever the graph of $u$ crosses one of the two hyperbola:
$$
h_1(r)=-\frac{n-1+\sqrt{n-1}}{(n-2)r}\, ,\quad h_2(r)=-\frac{n-1-\sqrt{n-1}}{(n-2)r}\, .
$$
Then rewrite \eq{equaz} as
\begin{align*}
\frac{1}{r^{n-1}}\Big(r^{n-1}u'(r)\Big)'&=(n-2)\left(u(r)+\frac{n-1+\sqrt{n-1}}{(n-2)r}\right)\left(u(r)+\frac{n-1-\sqrt{n-1}}{(n-2)r}\right)u(r)
\\&\,\,-(n-4)u(r)u'(r)+\left[u'(r)+\frac{u(r)}{r}\right]h(r)\, .
\end{align*}
If we assume that
\neweq{assumptions}
h(r)\ge0\quad\forall r\ge0\, ,\qquad n\ge4\, ,
\endeq
and that
\neweq{assu}
\exists R>0\mbox{ such that}\qquad u(R)=h_1(R)\, ,\quad u'(R)\le0\, ,
\endeq
then the above equation tells us that $r\mapsto r^{n-1}u'(r)$ is decreasing for $r>R$. In particular, we have that $u'(r)<0$ and $u(r)<h_1(r)$
for all $r>R$. Finally, this yields the existence of $\gamma>0$ such that
$$
\frac{1}{r^{n-1}}\Big(r^{n-1}u'(r)\Big)'\le\gamma u(r)^3\qquad\forall r>R\, .
$$
By arguing as in the proof of Theorem \ref{nonexists} (see Lemma \ref{MP} below) one obtains that
$$
\exists\overline{R}>R\quad\mbox{s.t.}\quad\lim_{r\to\overline{R}}u(r)=-\infty\, .
$$
Summarizing, we have

\begin{proposition}
If \eqref{assumptions} holds, then there exists no global solution of \eqref{equaz}-\eqref{ic} which satisfies \eqref{assu}.
\end{proposition}

\

\section{Existence results}

We start with a simple but interesting example. If
$$h(r)=-\frac{\alpha r}{2}\Big((n-2)\alpha r^2+n+2\Big)$$
then $u(r)=\alpha r$ solves \eq{equaz}-\eq{ic}. We point out that, in any case, the solution $u$ is {\em global and unbounded}. Moreover, if $\alpha>0$ then
$h(r)<0$ for all $r$ and the solution of \eq{equaz}-\eqref{ic} is positive, while if $\alpha<0$ then $h$ changes sign and the solution is negative.\par
Theorem \ref{nonexists} suggests that \eq{equaz} is more likely to have negative solutions whenever $h$ itself is negative. We prove that this is the case, at least in dimension $n=4$.

\begin{theorem}\label{existence} In dimension $n=4$, if $h(r)\leq 0$ for all $r>0$, then \eqref{equaz} admits infinitely many negative global solutions. More precisely, for any $\alpha<0$ the solution of \eqref{equaz}-\eqref{ic} is global and it satisfies $-\frac2r <u(r)<0$ for all $r>0$.
\end{theorem}

\

\section{Remarks and open problems}

We discuss some open problem related to conformal $f$-Yamabe metrics and to solutions of equation \eqref{cfypg}.

\medskip

\begin{enumerate}

\item In Theorem \ref{nonexists} we stated a partial nonexistence results for radial solutions in the Euclidean space, while Theorem \ref{existence} provides a general existence result. It would be interesting to prove a sharp condition on the potential function $f$ (or on its derivative) ensuring existence of global solutions to \eqref{equaz}.

\medskip

\item It is well known \cite{gnn} that global positive solution the Yamabe equation
$$
-\Delta u = u^{\frac{n+2}{n-2}}  \quad\text{on } \RR^n
$$
have to be radial (and thus classified). We could ask the same question for (general) solutions to \eqref{cfyp}. For a given $f\in C^{\infty}(\RR^n)$, are there any nonradial solutions $w$? If $f$ is radial, are all solutions to \eqref{cfyp} radial?

\medskip

\item In this paper we studied conformal $f$-Yamabe metrics for $(\RR^n,g_{\RR^n})$. What about other rotationally symmetric spaces? In particular, what we can say for the hyperbolic space $(\HH^n,g_{\HH^n})$ or the round sphere $(\SS^n,g_{\SS^n})$?

\medskip

\item In the existence result (Theorem \ref{existence}) the dimension $n=4$ seems to be peculiar, at least from the analytic point of view. Is there a geometric interpretation of this fact?

\end{enumerate}

\

\section{Proof of Theorem \ref{nonexists}}\label{proofnon}

Throughout this proof we will need the following particular class of test functions.

\begin{definition}\label{rho}
Let $\rho>0$. We say that a nonnegative function $\phi\in C^2_c[0,\infty)$ satisfies the $\rho$-property if
\neweq{prima}
\phi(r)=1\quad\mbox{for }r\in[0,\rho]\qquad\mbox{and}\qquad\phi(r)=0\quad\mbox{for }r\ge2\rho\, ,
\endeq
and if
\neweq{seconda}
\int^{2\rho}\frac{|\phi''(r)+\frac{2(n-1)}{r}\phi'(r)|^{3/2}}{\sqrt{\phi(r)}}\, dr<\infty\, .
\endeq
\end{definition}

It is clear that such functions exist; to see this, it suffices to replace any $\phi$ satisfying \eq{prima} with a power $\phi^k$ for
$k$ sufficiently large so that \eq{seconda} will be satisfied.\par
For the proof of Theorem \ref{nonexists}, we first observe that, since $\alpha>0$ the solution of \eq{equaz}-\eq{ic} is positive and
strictly increasing in a right neighborhood of $r=0$, say in some maximal interval $(0,R)$. Clearly, among $u$ and $u'$ the first one which can vanish is $u'$.
But if $u'(R)=0$ then, using the lower bound for $h$, we see that \eq{equaz} yields
$$
u''(R)\ge(n-2)u(R)^3+\frac{2(n-1)}{R}u(R)^2+\big(\frac{n-1}{R^2}+\frac{h(R)}{R}\big)u(R)\ge(n-2)u(R)^3+\frac{2(n-1)}{R}u(R)^2>0\, ,
$$
giving a contradiction. Therefore, $u'$ cannot vanish and two cases may occur:
\neweq{cases}
(i)\ R=\infty\, ,\qquad(ii)\ R<\infty\mbox{ and }\lim_{r\uparrow R}u(r)=+\infty\, .
\endeq
The proof will be complete if we show that $(ii)$ occurs. At this point, we distinguish two cases.\par\medskip
\noindent$\bullet$ {\bf Case }$n\in\{3,4\}$.\par\smallskip
In order to prove $(ii)$ in \eq{cases}, we argue for contradiction by assuming that $R=\infty$ so that $u,u'>0$ for all $r>0$. From the assumptions
and \eq{equaz}, we then infer that (recall $n\le4$)
\neweq{ineq}
u''(r)+\frac{2n-2}{r}u'(r)>(n-2)u(r)^3>0\qquad\forall r>0\, .
\endeq
To reach a contradiction we need the following estimate, inspired to the method developed by Mitidieri-Poho\v zaev \cite{mp} (see also the proof
of \cite[Proposition 5]{gazgru}).

\begin{lemma}\label{MP}
Assume that $w\in C^2[0,\infty)$. Then for any $\eps>0$, for any $\rho>0$ and for all $\phi$ satisfying the $\rho$-property, we have
\begin{align*}
\int_0^{2\rho}r^{2n-2}\left|w''(r)+\frac{2n-2}{r}w'(r)\right|\phi(r)\, dr&\le\frac{\eps}{3}\int_0^{2\rho}r^{2n-2}|w(r)|^3\phi(r)\, dr
\\&\,\,+\frac{2}{3\sqrt{\eps}}\int_\rho^{2\rho}r^{2n-2}\frac{|\phi''(r)+\frac{2n-2}{r}\phi'(r)|^{3/2}}{\sqrt{\phi(r)}}\, dr\, .
\end{align*}
\end{lemma}
\begin{proof} In this proof we will use the Young inequality in the following form:
\neweq{young1}
\forall\eps>0\quad \forall a,b\ge0\qquad ab\le \frac{\eps a^3}{3}+\frac{2b^{3/2}}{3\sqrt{\eps}}\, .
\endeq

Fix $\eps>0$ and $\rho>0$. We use a PDE approach and introduce the radial $C^2(\R^{2n-1})$-function $v$ such that $v(x)=w(|x|)$ for
all $x$: note that the space dimension is here $2n-1$ and that $\Delta v(x)=w''(|x|)+\frac{2n-2}{|x|}w'(|x|)$. Then, multiply $\Delta v$ by some
function $\Phi(x)=\phi(|x|)$, where $\phi$ satisfies the $\rho$-property. Since $\Phi\equiv1$ in $B_\rho$, two integration by parts and \eq{young1} yield
\begin{align*}
\int_{B_{2\rho}}\Delta v\Phi&=\int_{B_{2\rho}}v\Delta\Phi=\int_{B_{2\rho}\setminus B_\rho}v\Delta\Phi\, .
=\int_{B_{2\rho}\setminus B_\rho}v\Phi^{1/3}\, \frac{\Delta\Phi}{\Phi^{1/3}}\\&\le\frac{\eps}{3}\int_{B_{2\rho}}|v|^3\Phi+
\frac{2}{3\sqrt{\eps}}\int_{B_{2\rho}\setminus B_\rho}\frac{|\Delta\Phi|^{3/2}}{\Phi^{1/2}}
\end{align*}
and back to the radial form of $v$ and $\Phi$ this proves the statement.\end{proof}

Take a function $\phi_1$ satisfying the $1$-property and observe that the function
$$\phi_\rho(r):=\phi_1\left(\frac{r}{\rho}\right)\qquad\forall\rho>1$$
satisfies the $\rho$-property. Therefore, for all $\eps>0$, from \eq{ineq} and Lemma \ref{MP} we infer that
\begin{align*}
(n-2)\int_0^\rho r^{2n-2}u(r)^3\, dr&\le(n-2)\int_0^{2\rho} r^{2n-2}u(r)^3\phi_\rho(r)\, dr\\&\le
\int_0^{2\rho}r^{2n-2}\Big(u''(r)+\frac{2n-2}{r}u'(r)\Big)\phi_\rho(r)\, dr
\end{align*}
$$
\le\frac{\eps}{3}\int_0^{2\rho}r^{2n-2}u(r)^3\phi_\rho(r)\, dr
+\frac{2}{3\sqrt{\eps}}\int_\rho^{2\rho}r^{2n-2}\frac{|\phi_\rho''(r)+\frac{2n-2}{r}\phi_\rho'(r)|^{3/2}}{\sqrt{\phi_\rho(r)}}\, dr\, .
$$
Take $0<\eps<3(n-2)$, then the latter inequality yields
$$
\left(n-2-\frac{\eps}{3}\right)\int_0^\rho r^{2n-2}u(r)^3\, dr\le
\frac{2}{3\sqrt{\eps}}\int_\rho^{2\rho}r^{2n-2}\frac{|\phi_\rho''(r)+\frac{2n-2}{r}\phi_\rho'(r)|^{3/2}}{\sqrt{\phi_\rho(r)}}\, dr\, .
$$
With the change of variable $r=\rho t$ this becomes
$$
\left(n-2-\frac{\eps}{3}\right)\int_0^1 t^{2n-2}u(\rho t)^3\, dt\le\frac{2}{3\sqrt{\eps}}\frac{1}{\rho^3}\int_1^2 t^{2n-2}
\frac{|\phi_1''(t)+\frac{2n-2}{t}\phi_1'(t)|^{3/2}}{\sqrt{\phi_1(t)}}\, dt\ ,
$$
Since $u$ is increasing on $\R_+$, we have $u(\rho t)\ge u(t)$ for all $\rho>1$ so that the left hand side of this inequality is positive and increasing
for $\rho\ge1$. By letting $\rho\to\infty$, the right hand side tends to $0$ and this leads to a contradiction which rules out case $(i)$. Hence, case
$(ii)$ occurs and the solution $u$ of \eq{equaz}-\eq{ic} with $\alpha>0$ cannot be continued to all the interval $[0,\infty)$. This completes the proof
of Theorem \ref{nonexists} in the case $n=3,4$.\par\medskip
\noindent$\bullet$ {\bf Case }$n\ge5$.\par\smallskip
The same arguments leading to \eq{ineq} now yield
\neweq{ineq2}
u''(r)+\frac{2n-2}{r}u'(r)+\frac{n-4}{2}\Big(u(r)^2\Big)'>(n-2)u(r)^3>0\qquad\forall r>0
\endeq
and we need to estimate one more term. The companion of Lemma \ref{MP} reads

\begin{lemma}\label{MP2}
Assume that $w\in C^2[0,\infty)$. Then for any $\delta>0$, for any $\rho>0$ and for all $\phi$ satisfying the $\rho$-property, we have
$$
\int_0^{2\rho}r^{2n-2}\Big(w(r)^2\Big)'\phi(r)\, dr\le\frac{2\sqrt{\delta}}{3}\int_0^{2\rho}r^{2n-2}|w(r)|^3\phi(r)\, dr+
\frac{1}{3\delta}\int_\rho^{2\rho}r^{2n-2}\frac{|\phi'(r)|^3}{\phi(r)^2}\, dr\, .
$$
\end{lemma}
\begin{proof} In this proof we will use the Young inequality in the following form:
\neweq{young2}
\forall\delta>0\quad \forall a,b\ge0\qquad ab\le \frac{2\sqrt{\delta}\, a^{3/2}}{3}+\frac{b^3}{3\delta}\, .
\endeq

Fix $\delta>0$ and $\rho>0$; then take $\phi$ satisfying the $\rho$-property. An integration by parts yields
$$\int_0^{2\rho}r^{2n-2}\Big(w(r)^2\Big)'\phi(r)\, dr=-\int_0^{2\rho}w(r)^2\Big(2(n-1)r^{2n-3}\phi(r)+r^{2n-2}\phi'(r)\Big)\, dr$$
and, since $\phi\ge0$,
\begin{eqnarray*}
\int_0^{2\rho}r^{2n-2}\Big(w(r)^2\Big)'\phi(r)\, dr &\le& -\int_0^{2\rho}r^{2n-2}w(r)^2\phi'(r)\, dr\\
&\le& \int_0^{2\rho}\Big(r^{4(n-1)/3}w(r)^2\phi(r)^{2/3}\Big)\cdot\left(r^{2(n-1)/3}\frac{|\phi'(r)|}{\phi(r)^{2/3}}\right)\, dr\\
\mbox{[by \eq{young2}] }\ &\le& \frac{2\sqrt{\delta}}{3}\int_0^{2\rho}r^{2n-2}|w(r)|^3\phi(r)\, dr+
\frac{1}{3\delta}\int_0^{2\rho}r^{2n-2}\frac{|\phi'(r)|^3}{\phi(r)^2}\, dr\, .
\end{eqnarray*}
Since $\phi'\equiv0$ on $(0,\rho)$, this completes the proof.\end{proof}

Take again a function $\phi_1$ satisfying the $1$-property and let $\phi_\rho(r):=\phi_1(r/\rho)$ for all $\rho>1$
so that $\phi_\rho$ satisfies the $\rho$-property. Multiply \eq{ineq2} by $\phi_\rho$ and integrate over $(0,2\rho)$ to obtain
\begin{eqnarray*}
(n\!-\!2)\int_0^\rho r^{2n-2}u(r)^3\, dr &\le& (n\!-\!2)\int_0^{2\rho} r^{2n-2}u(r)^3\phi_\rho(r)\, dr\\
&\le& \int_0^{2\rho}r^{2n-2}\Big(u''(r)+\frac{2n-2}{r}u'(r)\Big)\phi_\rho(r)\, dr \\
&&\,\,+\frac{n-4}{2}\int_0^{2\rho}r^{2n-2}\Big(u(r)^2\Big)'\phi_\rho(r)\, dr\, .
\end{eqnarray*}

Then, by Lemmas \ref{MP} and \ref{MP2} we infer that
$$
\left(n-2-\frac{\eps}{3}-\frac{(n-4)\sqrt{\delta}}{3}\right)\int_0^\rho r^{2n-2}u(r)^3\, dr
$$
$$
\le\frac{2}{3\sqrt{\eps}}\int_\rho^{2\rho}r^{2n-2}\frac{|\phi_\rho''(r)+\frac{2n-2}{r}\phi_\rho'(r)|^{3/2}}{\sqrt{\phi_\rho(r)}}\, dr
+\frac{n-4}{6\delta}\int_\rho^{2\rho}r^{2n-2}\frac{|\phi_\rho'(r)|^3}{\phi_\rho(r)^2}\, dr\, .
$$
Take $\eps>0$ and $\delta>0$ sufficiently small in such a way that $C_{\eps,\delta}:=n-2-\frac{\eps}{3}-\frac{(n-4)\sqrt{\delta}}{3}>0$ and
perform the change of variable $r=\rho t$ to obtain
$$
C_{\eps,\delta}\int_0^1 t^{2n-2}u(\rho t)^3\, dt\le
\frac{2}{3\sqrt{\eps}}\frac{1}{\rho^3}\int_1^2 t^{2n-2}\frac{|\phi_1''(t)+\frac{2n-2}{t}\phi_1'(t)|^{3/2}}{\sqrt{\phi_1(t)}}\, dt
+\frac{n-4}{6\delta}\frac{1}{\rho^3}\int_1^2 t^{2n-2}\frac{|\phi_1'(t)|^3}{\phi_1(1)^2}\, dt\, .
$$
Since $u$ is increasing on $\R_+$, we have $u(\rho t)\ge u(t)$ for all $\rho>1$ so that the left hand side of this inequality is positive and increasing
for $\rho\ge1$. By letting $\rho\to\infty$, the right hand side tends to $0$ and this leads to a contradiction which rules out case $(i)$. Hence, case
$(ii)$ occurs and the solution $u$ of \eq{equaz}-\eq{ic} with $\alpha>0$ cannot be continued to all the interval $[0,\infty)$. This completes the proof
of Theorem \ref{nonexists} also in the case $n\ge5$.

\

\section{Proof of Theorem \ref{existence}}

For our convenience, we introduce the functions
$$
h(y)=(y+2)(y+1)y\, ,\qquad H(y)=\int_0^yh(\xi)d\xi=\frac{(y+2)^2y^2}{4}\, .
$$
Let $v(r)=ru(r)$, then $v$ satisfies the equation
\neweq{equv}
v''(r)+\left[\frac{n-3}{r}+\frac{n-4}{r}v(r)-h(r)\right]v'(r)=\frac{n-2}{r^2}\, h\big(v(r)\big)\, .
\endeq

If $u$ satisfies \eq{equaz}-\eq{ic} with $\alpha<0$, then $u(r)$ and $u'(r)$ are strictly negative in a right neighborhood of $r=0$.
By definition of $v$, also $v(r)$ is strictly negative in a right neighborhood of $r=0$.
We claim that $-2<v(r)<0$ for all $r>0$. If not, let $R>0$ be the first time where
\neweq{either}
\mbox{either }v(R)=0\mbox{ or }v(R)=-2\, .
\endeq
Multiply \eq{equv} by $r^2v'(r)$ and integrate over $[0,R]$ to obtain
$$
\int_0^R\Big(r^2v''(r)v'(r)+\big[n-3+(n-4)v(r)-rh(r)\big]rv'(r)^2\Big)\, dr=0
$$
since $H(0)=H(-2)=0$. An integration by parts then yields
$$
\int_0^R\big[n-4+(n-4)v(r)-rh(r)\big]rv'(r)^2\, dr+\frac{R^2v'(R)^2}{2}=0\, .
$$
If $n=4$ and $h(r)\leq 0$ we get a contradiction which shows that $R$ does not exist and therefore $-2<v(r)<0$ for all $r>0$. This proves the claim. Hence, by \eqref{equv}, also $v'$ and $v''$ remain bounded and the solutions exists. This concludes the proof of Theorem \ref{existence}.

\

\

\begin{ackn} The authors are members of the Gruppo Nazionale per l'Analisi Matematica, la Probabilit\`{a} e le loro Applicazioni (GNAMPA) of the Istituto Nazionale di Alta Matematica (INdAM). The first author is supported by the PRIN project ``Variational methods, with applications to problems in mathematical physics and geometry''. The second author is supported by the PRIN project ``Equazioni alle derivate parziali di tipo ellittico e parabolico: aspetti
geometrici, disuguaglianze collegate, e applicazioni''.
\end{ackn}

\

\

\end{document}